\documentclass[10pt, reqno]{amsart}
\usepackage{graphicx, amssymb, amsmath, amsthm}
\usepackage{comment}
\usepackage{epsfig}  
\usepackage{xcolor}
\usepackage{enumerate}
\usepackage{hyperref}
\usepackage{color}

\numberwithin{equation}{section}

\let\Re=\undefined\DeclareMathOperator*{\Re}{Re}
\let\Im=\undefined\DeclareMathOperator*{\Im}{Im}

\def\ov{\overline}

\newcommand{\pa}{\parallel}
\newcommand{\pe}{\perp}

\newtheorem{theorem}{Theorem}[section]
\newtheorem{lemma}[theorem]{Lemma}
\newtheorem{corollary}[theorem]{Corollary}
\newtheorem{proposition}[theorem]{Proposition}
\newtheorem{definition}[theorem]{Definition}
\newtheorem{remark}[theorem]{Remark}

\begin{document}

\title[Stabilization of a Hyperbolic Stokes type system]{On the Stabilization of a Hyperbolic Stokes system Under Geometric Control Condition}

\author[F. W. Chaves-Silva]{Felipe W. Chaves-Silva$^*$}
\address{Departament of Mathematics, Federal University of Para\'iba, CEP: 58051-900, Jo\~ao Pessoa, PB, Brazil}
\email{fchaves@mat.ufpb.br}
\author[C. Sun]{Chenmin Sun$^*$}
\address{CY Cergy-Paris Universit\'e, Laboratoire de math\'ematique AGM, UMR CNRS 8088, Cergy-Pontoise, France}
\email{chenmin.sun@u-cergy.fr}
\thanks{$^*$F. W. Chaves-Silva and Chenmin Sun were supported  by  the ERC project number 320845: Semi Classical Analysis of  Partial Differential Equations, ERC-2012-ADG}
\maketitle

\begin{abstract}
	In this article, we study the stabilization problem for a hyperbolic type Stokes system posed on a bounded domain. We show that when the damping effects are restricted to  a subdomain satisfying the geometric control condition, the energy of the system decays exponentially. The result is a consequence of a new quasi-mode estimate for the Stokes system.
\end{abstract}
\section{Introduction and Main Results}

Let $\Omega \subset \mathbb{R}^d$ ($d\geq 2$) be a bounded connected open set whose boundary $\partial \Omega$ is regular enough, $\omega$  be a small subset of $\Omega$ and let $T > 0$. In this article, we are interested in the stabilization problem for the following hyperbolic Stokes system:

\begin{equation}\label{WStokesmain}
\left \{   
\begin{array}{ll}
\partial^2_tu-\Delta u+\nabla p+a(x)\partial_t u=0 &  \mbox{in}  \   \mathbb{R}\times \Omega,  \\
\mathrm{div } \ u  = 0 & \mbox{in} \  \mathbb{R}\times \Omega, \\
u = 0 & \mbox{on} \  \mathbb{R}\times \partial \Omega , \\
(u(0, x),\partial_tu(0, x))=(u_0,v_0)\in V\times H,
\end{array}
\right. 
\end{equation}
where $V$ and $H$ are the usual spaces in the context of fluid mechanics:

 $$V=\{u\in H_{0}^1(\Omega)^d: \mathrm{div } \ u=0\}$$ and
$$H=\{u\in L^2(\Omega)^d: \mathrm{div } \ u =0 ,u\cdot\mathbb{\nu}|_{\partial\Omega}=0\},
$$
and $\nu(x)$ is the outward normal to $\Omega$ at the point $x \in \partial \Omega$. In \eqref{WStokesmain}, the  damping term $a\in L^{\infty}(\Omega)$ and satisfies $a(x)\geq 0$, for all $x\in \Omega$.

If $u= u(x,t)$ is a (sufficiently smooth) solution of the system,  we define its energy as
$$ E[u](t)=\frac{1}{2}\int_{\Omega}(|\partial_t u(t,x)|^2+|\nabla u(t,x|^2)dx, \ \ \ \forall t\in \mathbb{R},
$$
and when there is  no damping, namely $a\equiv 0$, the energy is conserved, while in general we only have that $E[u](t)$ is non-increasing: 
$$ \frac{dE[u]}{dt}=
-\int_{\Omega}a(x)|\partial_t u(t,x)|^2dx\leq 0.
$$

As for other hyperbolic systems, the stabilization problem  for \eqref{WStokesmain} concerns the decay rate in time of the energy $E[u](t)$ under appropriate assumptions on the damping term.

It is well-known that stabilization problems are closely related to observability and exact controllability problems in abstract settings. In fact, if we consider the undamped system
\begin{equation}\label{undampedWS}
\left \{   
\begin{array}{ll}
\partial^2_tu-\Delta u+\nabla p=0 &  \mbox{in}  \   \mathbb{R}\times \Omega,  \\
\mathrm{div } \ u  = 0 & \mbox{in} \  \mathbb{R}\times \Omega, \\
u = 0 & \mbox{on} \  \mathbb{R}\times \partial \Omega , \\
(u(0, x),\partial_tu(0, x))=(u_0,v_0)\in V\times H,
\end{array}
\right. 
\end{equation}
we say that \eqref{undampedWS} is observable at time $T$ with observation in $\omega$ if there exists $C>0$ such that 
\begin{equation} \label{observability0}
||u_0||_{V}^2 + ||v_0||^2_{H} \leq C\int_0^T\int_{\omega}|\partial_t u(t,x)|^2dxdt,
\end{equation}
for every $(u_0,v_0)\in V\times H$.

When \eqref{observability0} holds, one can show that for any $(u_0,v_0)\in V\times H$ there exists $f\in L^2((0,T) \times \omega)^d$ such that the solution of 
\begin{equation}\label{ControlWS}
\left \{   
\begin{array}{ll}
\partial^2_t u-\Delta u+\nabla p=f\mathbf{1}_{\omega} &  \mbox{in}  \   \mathbb{R}\times \Omega,  \\
\mathrm{div } \ u  = 0 & \mbox{in} \  \mathbb{R}\times \Omega, \\
u = 0 & \mbox{on} \  \mathbb{R}\times \partial \Omega , \\
(u(0, x),\partial_tu(0, x))=(u_0,v_0)\in V\times H,
\end{array}
\right. 
\end{equation}
satisfies
$$
u(T,x)=0, \ \ \partial_t u(T,x) =0,
$$
that is to say, system \eqref{ControlWS} is exact controllable at $T$ with control localized in $\omega$. Nevertheless, it is important to mention that a complete characterization of the sets $\omega$ for which \eqref{observability0} is true remains  open. A partial answer to this question was given by the first author in \cite{Felipe}.

The motivation for studying the stabilization of system \eqref{WStokesmain} is two-fold. First, system \eqref{undampedWS} is the hyperbolic counterpart of Stokes system, which is the linearized version of the well-known Navier-Stokes equation in fluid mechanics. In fact, if we know that  system \eqref{undampedWS} is exact controllable at some time $T>0$, with control applied to some control region $\omega$,  then the so-called Control Transmutation Method can be applied to obtain the null controllability at any time and the optimal cost of controllability (in time) for the Stokes system (for more details, see \cite{Felipe}). On the other hand, system \eqref{undampedWS} comes from simple models of dynamical elasticity for incompressible materials. More precisely, it can be derived as a limit model of the Lam\'{e} system in the linear elastic theory when one parameter tends to infinity (\cite{JLions}). For the sake of completeness, in Appendix we give a derivation of the system \eqref{undampedWS} from the Lam\'e system. It is important to remark that the stabilization problem for the Lam\'{e} system has been already studied in \cite{BL}.

To state our main result, let us introduce several concepts. Some terminologies and notations will be clear in the next section.

\begin{definition}
	We say that the support of a non-negative function $a\in C(\ov{\Omega})$ satisfies the geometric control condition (GCC in short) if there exists $T>0$, such that each generalized bicharacteristic ray $\gamma(t)$ with speed $1$ issued from a point $\rho\in ^bT^{*}\ov{\Omega}$ enters the set $\{x\in\ov{\Omega}:a(x)>0\}$ in a time $t<T$. 
\end{definition}

We recall that an open set $\Omega$ has no infinite order of contact, if there is no geodesic tangent to $\partial\Omega$ of infinite order. More precisely, in the decomposition
$$ T^*\partial\Omega=\mathcal{E}\cup \mathcal{H}\cup \mathcal{G},
$$
we have 
$$ \mathcal{G}=\bigcup_{j=2}^{\infty}\mathcal{G}^j.
$$
Here, the sets $\mathcal{E},\mathcal{H},\mathcal{G}$ are called the elliptic region, the hyperbolic region and the glancing surface, respectively, and $\mathcal{G}^j$ is the set of points in $\mathcal{G}$ with $j-$th order of contact. The precise definition will be given in the next section.

The first main result in this article is as follows: 
\begin{theorem}\label{main}
	Suppose that $\Omega\subset \mathbb{R}^d$ is a bounded open set with no infinite order of contact and $a\in C(\ov{\Omega})$ is a non-negative function whose  support satisfies the geometric control condition. Then, there exist positive constants $C_0$ and $\alpha$ such that for any $(u_0,v_0)\in V\times H$, the energy of the corresponding solution $u(t)$ to \eqref{WStokesmain} decays exponentially:
	\begin{equation}
	\label{stab}
	E[u](t)\leq C_0E[u](0)e^{-\alpha t}
	\end{equation}
	for all $t\geq 0$.
\end{theorem}

In what follows, we say that the stabilization of \eqref{WStokesmain} holds if \eqref{stab} holds true.

\begin{remark}
	As a byproduct of the proof of Theorem \ref{main}, we obtain the null (exact) controllability at some time $T$ of  system \eqref{ControlWS}. Namely, there exist $T>0$ and a control $f\in L^2([0,T]\times\omega)$ such that the corresponding solution $u$ to \eqref{ControlWS} satisfies $(u(T),\partial_tu(T))=(0,0).$ However, we do not describe the control time $T$ explicitly, since we prove the observability inequality \eqref{observability0} by reducing it to a quasi-mode estimate. 
\end{remark}

Let us mention that if $a$ is supported in a neighborhood of the boundary $\partial\Omega$, the same result is true by adapting the strategy in \cite{Felipe}, where the author has proved the exact controllability of the system \eqref{ControlWS} with $\omega$ be a neighborhood of $\partial\Omega$. Our result is a generalization of the result in \cite{Felipe}.

The pioneering work of J.~Rauch and M.~Taylor \cite{RT} related the exponential decay of damped wave equations to the geometric control condition (GCC) of damped region on compact Riemannian manifold without boundary. Until the celebrated work of  Bardos-Lebeau-Rauch \cite{BLR}, the presence of the boundary has been understood and the exactly controllability for wave equations as well as the exponential stabilization are obtained under (GCC). The proof mainly relies on the propagation of singularities along the Melrose-Sj\"{o}strand flow. Later on, the tool of micro-local defect measure, introduced by G\'{e}rard and 
Tartar independently, has been used to simplify the proof of these results and to adapt to many other problems, see for example \cite{BL} for the Lam\'{e} systems and \cite{DRL} for a coupled wave system.
The key ingredient of the measure-based proof is the propagation formula, which can be viewed as a transport equation for the defect measure. This is a more quantitative version of the classical propagation of singularities, which usually describes the invariance of the wave front set along the Melrose-Sj\"ostrand flow. 

For the present system \eqref{WStokesmain}, the presence of the pressure term $\nabla p$ introduces nontrivial difficulties if we want to adapt the strategy in \cite{BL} directly, due to the rough regularity of time-dependent harmonic function $p(t,x)$. However, following \cite{Ge78}, the exponentially stabilization of \eqref{WStokesmain} can be reduced to the following semi-classical version of the observation inequality:
\begin{proposition}\label{semi-obser}
	Assume that $a\in L^{\infty}(\Omega)\cap C^0(\ov{\Omega})$ and $\int_{\Omega}adx>0$.  Then the stabilization of \eqref{WStokesmain} holds if the following statement is true:
	$$ \exists h_0>0,C>0 \textrm{ such that for all } 0<h<h_0, \text{ if } (u,q,f)\in H^2(\Omega)\cap V\times H^1(\Omega)\times H
	$$
	solves the equation
	\begin{equation}\label{quasimode}
	-h^2\Delta u-u+h\nabla q =f,
	\end{equation}
	then the inequality 
	\begin{equation}
	\label{semi-ob}
	\|u\|_{L^2(\Omega)}\leq C\Big(\|a^{1/2}u\|_{L^2(\Omega)}+\frac{1}{h}\|f\|_{L^2(\Omega)}\Big).
	\end{equation}
	holds.

\end{proposition}

Note that the system \eqref{semi-ob} is just a quasi-mode equation of the stationary Stokes system. In particular, if $f=0$, the solution $u(h)$ is a eigenfunction of the Stokes operator corresponding to the eigenvalues $h^{-2}$. 

The proof of \eqref{semi-ob} is based on the propagation of semi-classical measure $\mu$ in the recent work \cite{Sun} of the second author. We give a brief recall here. The sequence of pressures $q(h)$ are harmonic, and their impact on the solution only occurs at the boundary. It has been shown that the measure is propagated along bicharacteristic rays which is invariant under the Melrose-Sj\"ostrand flow. When a ray reaches the boundary, careful analysis between the wave-like propagation phenomenon and the impact of the pressure allows us to deduce the propagation of the support of the measure $\mu$ along generalized bicharacteristic rays defined in \cite{MS}.

We organize this article as follows. In Section 2, we give some notations, definitions and classical results. In Section 3, we follow the strategy in \cite{BG} to reduce the stabilization to the semi-classical observability \eqref{semi-ob}. In Section 4, we prove the semi-classical observability by adapting the propagation result. Finally in Appendix, we give the derivation from the Lam\'{e} system to the hyperbolic Stokes system \eqref{WStokesmain}.

\section{Preliminary}
\subsection{Notations}
For a manifold $M$, we let $TM$ be its tangent bundle and $T^*M$ be the cotangent bundle with canonical projection
$$ \pi:TM( \textrm{ or }T^*M)\rightarrow M.
$$

In the tubular neighborhood of boundary, we can identify $\Omega$ locally as $[0,\epsilon_0)\times X$, $X=\{x'\in\mathbb{R}^{d-1}:|x'|<1\}$. For $x\in\ov{\Omega}$, we note $x=(y,x')$, where $y\in [0,\epsilon_0),x'\in X$, and $x\in\partial\Omega$ if and only if $x=(0,x')$. In this coordinate system, the Euclidean metric $dx^2$ can be written as matrices
\begin{equation*}
g=\left(
\begin{array}{cc}
1 & 0 \\
0 & M(y,x') \\
\end{array}
\right),
g^{-1}=\left(
\begin{array}{cc}
1 & 0 \\
0 & \alpha(y,x') \\
\end{array}
\right),
\end{equation*}
with $|\xi'|_{\alpha(y,x'))}^2=\langle\xi',\alpha(y,x')\xi'\rangle_{\mathbb{C}^{d-1}}$ be the induced metric on $T^*\partial\Omega$, parametrized by $y$. Note that
$|\xi'|_{\alpha(0,x')}^2=\langle\xi',\alpha(0,x')\xi'\rangle_{\mathbb{C}^{d-1}}$ is the natural norm on $T^*\partial\Omega$, dual of the norm on $T\partial\Omega$, induced by the canonical metric on $\ov{\Omega}$. Write $(x,\xi)=(y,x',\eta,\xi')$ and denote by $|\xi|$ the Euclidean norm on $T^*\mathbb{R}^d$.

We define the $L^2$ norms and inner product on $[0,\epsilon_0)\times X$ via
$$ \|u\|_{L_{y,x'}^2}^2:=\int_0^{\epsilon_0}\int_X |u|^2d_{g(y,\cdot)}x'dy,
$$
$$ (u|v)_{L_{y,x'}^2}:=(u|v)_{\Omega}:=\int_0^{\epsilon_0}\int_X u\cdot \ov{v}d_{g(y,\cdot)}x'dy,
$$
$$\|u(y,\cdot)\|_{L_{x'}^2}^2:=\int_X|u(y,\cdot)|^2d_{g(y,\cdot)}x',
$$
$$ (u|v)_{L_{x'}^2}(y):=\int_X u(y,\cdot)\cdot\ov{v(y,\cdot)}d_{g(y,\cdot)}x',
$$
where the measure $d_{g(y,\cdot)}x'$ is the induced measure on $X$, parametrized by $y\in[0,\epsilon_0)$ such that $d_{g(y,\cdot)}x'dy=\mathcal{L}(dx)$, the Lebesgue measure on $\mathbb{R}^d$. Note that the measure $d_{g(0,\cdot)}x'$ is nothing but the surface measure on $\partial\Omega$. In certain situations we perform using global notation for inner product:
\begin{gather*}
(u|v)_{\Omega}:=\int_{\Omega}u\cdot\ov{v}dx,\\
(f|g)_{\partial\Omega}:=\int_{\partial\Omega}f\cdot\ov{g}d\sigma(x)
\end{gather*}
In the tubular neighborhood, we can write a vector field $X=(X_{\pa},X_{\pe})$, where $X_{\pa}$ stands for the components parallel to the boundary while $X_{\pe}$ stands for the normal component with the following convention: $(0,a)=-a\nu$. 

As in \cite{FL}, we will write down system \eqref{WStokesmain} in the tubular neighborhood. For
$u=(u_{\pa},u_{\pe})$, equation \eqref{WStokesmain} can be rewritten:
\begin{equation}
\left\{
\begin{aligned}
& (-h^2\Delta_{\pa}-1)u_{\pa}+h\nabla_{x'}q=f_{\pa},\\
& (-h^2\Delta_{g}-1)u_{\pe}+h\partial_y q=f_{\pe},\\
&h\textrm{ div }_{\pa}u_{\pa}+\frac{h}{\sqrt{\det g}}\partial_y(\sqrt{\mathrm{det }g} u_{\pe})=0
\end{aligned}
\right.
\end{equation}
where
$$ h^2\Delta_{\pa}=h^2\partial_y^2-\Lambda^2(y,x',hD_{x'})+hM_{\pa}(y,x',hD_x')+hM_1(y,x')h\partial_y,
$$
$$ h^2\Delta_{g}=h^2\partial_y^2-\Lambda^2(y,x',hD_{x'})+hM_{\pe}(y,x',hD_x')+
hN_1(y,x')h\partial_y,
$$
$$ h\textrm{ div }_{\pa}u_{\pa}=\frac{h}{\sqrt{\det g}}\sum_{j=1}^{N-1}\partial_{x'_j}(\sqrt{\det g}u_{\pa,j}).
$$
Note that $h^2\Lambda^2(y,x',hD_{x'})$ has the symbol $\lambda^2=|\xi'|^2_{g(y,\cdot)}$, $M_{\pa,\pe}$ are both first-order matrix-valued semi-classical differential operators, and $M_1, N_1$ are zero-order matrix-valued functions.

\subsection{Geometric Preliminaries}
Denote by $^bT\ov{\Omega}$ the vector bundle whose sections are the vector fields $X(p)$ on $\ov{\Omega}$ with $X(p)\in T_p\partial\Omega$ if $p\in\partial\Omega$. Moreover, denote by $^bT^*\ov{\Omega}$ the Melrose's compressed cotangent bundle which is the dual bundle of $^bT\ov{\Omega}$. Let 
$$ j:T^*\ov{\Omega}\rightarrow ^bT^*\ov{\Omega}
$$
be the canonical map. In our geodesic coordinate system near $\partial\Omega$, $^bT\ov{\Omega}$ is generated by the vector fields $\frac{\partial}{\partial x'_1},\cdot\cdot\cdot, 
\frac{\partial}{\partial x'_{d-1}},y\frac{\partial}{\partial y}$ and thus $j$ is defined by
$$ j(y,x';\eta,\xi')=(y,x';v=y\eta,\xi').
$$

The principal symbol of operator $P_h=-(h^2\Delta+1)$ is $$p(y,x',\eta,\xi')=\eta^2+|\xi'|_{\alpha(y,x')}^2-1.$$
By Car($P$) we denote the characteristic variety of $p$:
$$\textrm{Car}(P):=\{(x,\xi)\in T^*\mathbb{R}^{d}|_{\ov{\Omega}}:p(x,\xi)=0\},
Z:=j(\textrm{Car}(P)).
$$
By writing in another way
$$p=\eta^2-r(y,x',\xi'), r(y,x',\xi')=1-|\xi'|_{\alpha}^2,
$$
we have the decomposition
$$ T^*\partial\Omega=\mathcal{E}\cup\mathcal{H}\cup\mathcal{G},
$$
according to the value of $r_0:=r|_{y=0}$ where
$$\mathcal{E}=\{r_0<0\},
\mathcal{H}=\{r_0>0\},
\mathcal{G}=\{r_0=0\}.
$$
The sets $\mathcal{E},\mathcal{H},
\mathcal{G}$ are called elliptic, hyperbolic and glancing, with respectively. 

For a symplectic manifold $S$ with local coordinate $(z,\zeta)$, a Hamiltonian vector field associated with a real function $f$ is given by
$$ H_f=\frac{\partial f}{\partial \zeta}\frac{\partial}{\partial z}-\frac{\partial f}{\partial z}\frac{\partial}{\partial\zeta}.
$$
Now for $(x,\xi)\in\Omega$ far away from the boundary, the Hamiltonian vector field associated to the characteristic function $p$ is given by
$$ H_p=2\xi\frac{\partial}{\partial x}.
$$
We call the trajectory of the flow
$$ \phi_s:(x,\xi)\mapsto (x+s\xi,\xi)
$$
bicharacteristic or simply ray, provided that the point $x+s\xi$ is still in the interior.

To classify different scenarios as a ray approaching the boundary, we need more accurate decomposition of the glancing set $\mathcal{G}$. Let $r_1=\partial_yr|_{y=0}$ and define
$$ \mathcal{G}^{k+3}=\{(x',\xi'):r_0(x',\xi')=0,H_{r_0}^j(r_1)=0,\forall j\leq k;H_{r_0}^{k+1}(r_1)\neq 0\}, k\geq 0
$$
$$
\mathcal{G}^{2,\pm}:=\{(x',\xi'):r_0(x',\xi')=0,\pm r_1(x',\xi')>0\},\mathcal{G}^2:=\mathcal{G}^{2,+}\cup\mathcal{G}^{2,-}.
$$
No infinite order of contact means that we can decompose $\mathcal{G}$ into
$$ \mathcal{G}=\bigcup_{j=2}^{\infty}\mathcal{G}^j.
$$

Given a ray $\gamma(s)$ with $\pi(\gamma(0))\in \Omega$ and $\pi(\gamma(s_0))\in\partial\Omega$ be the first point who attaches the boundary.  If $\gamma(s_0)\in\mathcal{H}$, then $\eta_{\pm}(\gamma(s_0))=\pm\sqrt{r_0(\gamma(s_0))}$ be the two different roots of $\eta^2=r_0$ at this point. Notice that the ray starting with direction $\eta_-$ will leave $\Omega$,  while the ray with direction $\eta_+$ will enter the interior of $\Omega$. This motivates the following definition of broken bicharacteristic:
\begin{definition}[\cite{Hor}]
	A broken bicharacteristic arc of $p$ is a map:
	$$ s\in I\setminus B\mapsto \gamma(s)\in T^*\Omega\setminus \{0\},
	$$
	where $I$ is an interval on $\mathbb{R}$ and $B$ is a discrete subset, such that
	\begin{enumerate}
		\item If $J$ is an interval contained in $I\setminus B$, then 
		$ s\in J\mapsto \gamma(s)
		$
		is a bicharacteristic of $P_h$ over $\Omega$.
		\item If $s\in B$, then the limits $\gamma(s^+)$ and $\gamma(s^-)$ exist and belongs to $T_x^*\ov{\Omega}\setminus \{0\}$ for some $x\in\partial\Omega$, and the projections in $T_x^*\partial\Omega\setminus \{0\}$ are the same hyperbolic point.
	\end{enumerate}
\end{definition}
When a ray $\gamma(s)$ arrives at some point $\rho_0\in\mathcal{G}$, there are several situations. If $\rho_0\in\mathcal{G}^{2,+}$, then the ray passes transversally over $\rho_0$ and enters $T^*\Omega$ immediately. If $\rho_0\in\mathcal{G}^{2,-}$ or $\rho_0\in\mathcal{G}^k$ for some $k\geq3$, then we can continue it inside $T^*\partial\Omega$
as long as it can not leave the boundary along the trajectory of the Hamiltonian flow of $H_{-r_0}$. We now give the precise definition.

\begin{definition}[\cite{Hor}]
	A generalized bicharacteristic ray of $p$ is a map:
	$$ s\in I\setminus B\mapsto \gamma(s)\in (T^{*}\ov{\Omega}\setminus T^*\partial\Omega) \cup \mathcal{G}
	$$
	where $I$ is an interval on $\mathbb{R}$ and $B$ is a discrete set of $I$ such that $p\circ \gamma=0$ and the following:
	\begin{enumerate}
		\item $\gamma(s)$ is differentiable and $\frac{d\gamma}{ds}=H_{p}(\gamma(s))$ if $\gamma(s)\in T^*\ov{\Omega}\setminus T^*\partial\Omega $ or $\gamma(s)\in\mathcal{G}^{2,+}$.
		\item Every $s\in B$ is isolated, $\gamma(s)\in T^*\ov{\Omega}\setminus T^*\partial\Omega$ if $s\neq t$ and $|s-t|$ is small enough, the limits $\gamma(s^{\pm})$ exist and are different points in the same fibre of $T^*\partial\Omega$.
		\item $\gamma(s)$ is differentiable and $\frac{d\gamma}{ds}=H_{-r_0}(\gamma(s))$ if $\gamma(s)\in \mathcal{G}\setminus \mathcal{G}^{2,+}$.
	\end{enumerate}
\end{definition}
\begin{remark}
	The definition above does not depend on the choice local coordinate, and in the geodesic coordinate system, the map
	$$ s\mapsto (y(s),\eta^2(s),x'(s),\xi'(s)) 
	$$
	is always continuous and 
	$$ s\mapsto (x'(s),\xi'(s))
	$$
	is always differentiable and satisfies the ordinary differential equations
	$$ \frac{dx'}{dt}=-\frac{\partial r}{\partial \xi'},\frac{d\xi'}{dt}=\frac{\partial r}{\partial x'},
	$$
	the map $s\mapsto y(s)$ is left and right differentiable with derivative $2\eta(s^{\pm})$ for any $s\in B$ (hyperbolic point).
	
	Moreover, there is also the continuous dependence with the initial data, namely the map
	$$ (s,\rho)\mapsto (y(s,\rho),\eta^2(s,\rho),x'(s,\rho),\xi'(s,\rho))
	$$
	is continuous. We denote the flow map by $\gamma(s,\rho)$.
\end{remark}

\begin{remark}
	Under the map $j:T^*\ov{\Omega}\rightarrow ^bT^*\ov{\Omega}$, one could regard $\gamma(s)$ as a continuous flow on the compressed cotangent bundle $^bT^*\ov{\Omega}$, and it is called the Melrose-Sj\"{o}strand flow. We will also call each trajectory generalized bicharacteristic or simply ray in the sequel. 
\end{remark}

It is well-known that if there is no infinite contact in $\mathcal{G}$, a generalized bicharacteristic is uniquely determined by any one of its points. In other words, the Melrose-Sj\"{o}strand flow is globally well-defined. See \cite{Hor} for more discussion.


\section{Review of Semi classical propagation of singularity}

\subsection{Definition of defect measure}
We follow closely \cite{NB1} here, see also \cite{PG1} for a little different but more comprehensive introduction.

Define the partial symbol class $S^m_{\xi'}$ and the class of boundary $h$-pseudo-differential operators $ \mathcal{A}_h^m$ as follows 
\begin{gather*}
S^m_{\xi'}:=\{a(y,x',\xi'):\sup_{\alpha,\beta,y\in[0,\epsilon_0]}|\partial_{x'}^{\alpha}
\partial_{\xi'}^{\beta}a(y,x',\xi')|\leq C_{m,\alpha,\beta}(1+|\xi'|)^{m-\beta}\}.\\
\mathcal{A}_h^m=:\textrm{Op}_h^{comp}(S^m)+\textrm{Op}_h(S_{\xi'}^m):=\mathcal{A}^m_{h,i}+\mathcal{A}^m_{h,,\partial}.
\end{gather*}

Denote by $U$ a tubular neighborhood of $\partial\Omega$. Consider functions of the form $a=a_i+a_{\partial}$ with $a_i\in C_c^{\infty}(\Omega\times\mathbb{R}^{d})$ which can be viewed as a symbol in $S^0$, and $a_{\partial}\in C_c^{\infty}(U\times\mathbb{R}^{d-1})$ can be viewed as a symbol in $S_{\xi'}^0$. We quantize $a$ via the formula (in local coordinate)
\begin{equation*}
\begin{split}
\textrm{Op}_h(a)f(y,x')=&
\frac{1}{(2\pi h)^d}\int_{\mathbb{R}^{2d}}
e^{\frac{i(x-z)\xi}{h}}a_i(x,\xi)f(z)dzd\xi\\
&+\frac{1}{(2\pi h)^{d-1}}
\int_{\mathbb{R}^{2(d-1)}}
e^{\frac{i(x'-z')\xi'}{h}}
a_{\partial}(y,x',\xi')f(y,z')dz'd\xi'.
\end{split}
\end{equation*}

Notice that the action of the tangential operator $\mathrm{Op}_h(a_{\partial})$ can be viewed as pseudo-differential operator on the manifold $\partial\Omega$, parametrized by the parameter $y\in[0,\epsilon_0)$. No doubt that the definition of the operator $\mathrm{Op}_h(a_{\partial})$ depends on the choice of local coordinate of $\partial\Omega$. However, the bounded family of operators $\mathcal{A}^m_{h,\partial}$ is defined uniquely up to a family of operators with norms uniformly dominated by $Ch$, as $h\rightarrow 0$. See \cite{PG1} for more details. Moreover, for any family $(A_h)$, such that 
$$ \|A_h-\textrm{Op}_h(a_{\partial})\|_{L^2\rightarrow L^2}=O(h),
$$
the principal symbol $\sigma(A)$ is determined uniquely as a function on $T^*\partial\Omega$, smoothly depending on $y$, i.e.
$\sigma(A)\in C^{\infty}([0,\epsilon_0)\times T^*\partial\Omega)$. 

When we deal with vector-valued functions, we could require the symbol $a$ to be matrix-valued. Now for any sequence of vector-valued function $w_k$, uniformly bounded in $L^2(\Omega)$, there exists a subsequence (still denoted $w_k$ for simplicity), and a nonnegative definite Hermitian matrix-valued measure $\mu_i$ on $T^*\Omega$ such that
$$ \lim_{k\rightarrow 0}(\textrm{Op}_{h_k}(a_i)w_k|w_k)_{L^2}=\langle\mu_i,a\rangle
:=
\int_{T^*\Omega}\textrm{tr }(ad\mu_i).
$$  
For a proof, see for example \cite{NB1}, and the micro-local version was appeared in \cite{PG2}.

From now on we will only deal with scalar-valued operator, even though we will encounter vector-valued functions in the analysis. Suppose $u_k$ be a sequence of solutions to \eqref{Stokes-semi}, under the assumptions below:
\begin{equation}\label{assumption1}
\begin{split}
&\|u_k\|_{L^2(\Omega)}=O(1),f_k\in H\textrm{ and }\|f_k\|_{L^2(\Omega)}=o(h_k),\\
&\|h\nabla q_k\|_{L^2(\Omega)}=O(1),\int_{\Omega}q_kdx=0,\\
\end{split}
\end{equation}

The following result shows that the interior measure $\mu_i$ is supported on the $\textrm{Car}(P)$.
\begin{proposition}\label{Car}
	Let $a_i\in C_c^{\infty}(\Omega\times\mathbb{R}^d)$ be equal to $0$ near $\mathrm{Car}(P)$, then we have
	$$ \lim_{k\rightarrow\infty}(\mathrm{Op}_{h_k}(a_i)u_k|u_k)_{L^2}= 0.
	$$ 
\end{proposition} 
\begin{proof} Note that the symbol $b(x,\xi)=\frac{a_i(x,\xi)}{|\xi|^2-1}\in S^0$ is well-defined from the assumption on $a_i$. From symbolic calculus, we have
	$$ \textrm{Op}_{h_k}(a_i)=B_{h_k}(-h_k^2\Delta-1)+O_{L^2\rightarrow L^2}(h_k).
	$$
	Therefore
	\begin{equation*}
	\begin{split}
	(B_{h_k}(-h_k^2\Delta-1)u_k|u_k)_{L^2}&=(B_{h_k}f_k|u_k)_{L^2}-(B_{h_k}h_k\nabla q_k|u_k)_{L^2}\\
	&=o(1)+([h_k\nabla,B_{h_k}]q_k|u_k)_{L^2}-(h_k\nabla B_{h_k}q_k|u_k)_{L^2}\\
	&=o(1), \textrm{ as }k\rightarrow\infty,
	\end{split}
	\end{equation*}
	where in the last line we have used the symbolic calculus, integrating by part, and Lemma \ref{press.norm}.
\end{proof}

Now we denote by $Z=j(\mathrm{Car}(P))$.
Proposition \ref{Car} indicates that the interior defect measure $\mu_i$ is supported on 
$Z$. To define the defect measure up to the boundary, we have to check that if 
$a_{\partial}\in C_c^{\infty}(U\times\mathbb{R}^{d-1})$ vanishing near $Z$ (i.e. $a_{\partial}$ is supported in the elliptic region for all $y$ small) then
$$ \lim_{k\rightarrow\infty}(\textrm{Op}_{h_k}(a_{\partial})u_k|u_k)_{L^2}=0.
$$
Indeed, this can be ensured by the analysis of boundary value problem in the elliptic region, and the reader can consult section 6. Now for any family of operator $A_h\in\mathcal{A}_h^0$, let $a=\sigma(A_h)$ be the principal symbol of $A_h$
and we define $\kappa(a)\in C^0(Z)$ via
$\kappa(a)(\rho):=a(j^{-1}(\rho)).$
Note that $Z$ is a locally compact metric space and the set
$$ \{\kappa(a):a=\sigma(A_h),A_h\in\mathcal{A}_h^0\}
$$
is a locally dense subset of $C^0(Z)$.
We then have the following proposition, which guarantees the existence of a Radon measure on $Z$:
\begin{proposition}
	There exists a subsequence of $u_k,h_k$ and a nonnegative definite Hermitian matrix-valued Radon measure $\mu$, such that
	$$ \lim_{k\rightarrow\infty}
	(A_{h_k}u_k|u_k)_{L^2}=\langle\mu,\kappa(a)\rangle, a=\sigma(A_{h}),\forall A_h\in \mathcal{A}_h^0.
	$$
\end{proposition} 
The proof of this result can be found in \cite{NB1}, see also \cite{BL} and \cite{PG2} for its micro-local counterpart. Notice that if we write 
$ a=a_i+a_{\partial},
$
then
$$
(A_k u_k| u_k)\rightarrow \int_{T^*\Omega}\textrm{ tr }(a_i(\rho)d\mu_i(\rho))+
\int_Z \textrm{ tr }(a_{\partial}(\rho)d\mu(\rho)).
$$

The following result shows that information about frequencies higher than the scale $h_k^{-1}$ is not lost, and the measure $\mu$ contains the relevant information of the sequence $(u_k)$.
\begin{proposition}[\cite{Sun}]\label{frequencyconcentrate}
The sequence of solution $(u_k)$ is $h_k-$oscillating in the following sense:
$$ \lim_{R\rightarrow\infty}\limsup_{k\rightarrow\infty}
\int_{|\xi|\geq Rh_k^{-1}}
|\widehat{\psi u_k}(\xi)|^2d\xi=0,\forall \psi\in C_c^{\infty}(\Omega),
$$
$$ \lim_{R\rightarrow\infty}\limsup_{k\rightarrow\infty}
\int_0^{\epsilon_0}dy\int_{|\xi'|\geq Rh_k^{-1}}
|\widehat{\psi u_k}(y,\xi')|^2d\xi'=0,\forall \psi\in C_c^{\infty}(\ov{\Omega}),
$$
where in the second formula,  the Fourier transform involved is only the $x'$ direction.
\end{proposition}

A direct consequence is the following:
\begin{corollary}\label{vanishing}
Suppose $a^{1/2}u_k\rightarrow 0$ in $L^2(\Omega)$, and $\mu$ is the defect measure associated with $(u_k,h_k)$, then
$$ \langle\mu,a\rangle=0.
$$  
\end{corollary}

\subsection{Recall of propagation theorem}
Now let us recall the several results proved in \cite{Sun}.

In the interior, the full transport property of defect measure is proved.
\begin{proposition}[\cite{Sun}]\label{interior propagation}
	For any real-valued scalar function $a\in C_c^{\infty}(\Omega\times\mathbb{R}^d)$ vanishing near $\xi=0$, we have
	$$ \frac{d}{ds}\langle\mu,a\circ \gamma(s,\cdot)\rangle=0.
	$$
\end{proposition}

The following proposition illustrates that near a elliptic point on the boundary, there is no accumulation of singularity.
\begin{proposition}[\cite{Sun}]\label{elliptic}
	$\mu\mathbf{1}_{\mathcal{E}}=0$.  If we let $\nu$ be the semi-classical defect measure of the sequence $(h_k\partial_{\mathbf{\nu}}u_k|_{\partial\Omega},h_k)$, then $\nu\mathbf{1}_{\mathcal{E}}=0$. 
\end{proposition}

When a ray travels near a hyperbolic point or a point on the glancing surface, the knowledge of the singularity is much less known. Nevertheless, we have 
\begin{theorem}[\cite{Sun}]\label{St}
	Assume that $\Omega$ is a smooth, bounded domain with no infinite order of contact on the boundary. Suppose $(u_k)$ is a sequence of solutions to the quasi-mode problem \eqref{WStokesmain} with semi-classical parameters $h=h_k$. Assume that $f_k\in H$, $\|f_k\|_{L^2(\Omega)}=o(h_k)$ and  $u_k$ converges weakly to $0$ in $L^2(\Omega)$. Assume that $\mu$ is any semi-classical measure associated to some subsequence of $(u_k,h_k)$, then $\mathrm{supp }\mu$ is invariant under Melrose-Sj\"{o}strand flow. 	
\end{theorem}


\section{Reduction to Semi-classical observability}

This section is devoted to the proof of Proposition \ref{semi-obser}. In fact, it is classical from \cite{Ge78} 
that stabilization
or observability of a self adjoint evolution system is equivalent to resolvent estimates. See also \cite{BG}, \cite{Burq-Zworski}.

Recall that the damped system is given by
\begin{equation}\label{WStokes1}
\left\{
\begin{aligned}
&\partial^2_tu-\Delta u+a(x)\partial_tu+\nabla p=0,(t,x)\in \mathbb{R}\times \Omega\\
&\mathrm{div } u=0,\textrm{in} \ \Omega \\
&u(t,.)|_{\partial\Omega}=0 \\
&(u(0),\partial_tu(0))=(u_0,v_0)\in V\times H
\end{aligned}
\right.
\end{equation}

We always assume that $\Omega\subset \mathbb{R}^d$ is a bounded domain (open, connected set). We use $\mathbb{\nu}$ to denote the outward normal vector on $\partial\Omega$ and the damping term
$a\in L^{\infty}(\Omega)$ with $a(x)\geq 0$.

We also consider the undamped system
\begin{equation}\label{WStokes-NDamp}
\left\{
\begin{aligned}
&\partial^2_tu-\Delta u+\nabla p=0,(t,x)\in \mathbb{R}\times \Omega\\
&\mathrm{div } u=0,\textrm{in} \ \Omega \\
&u(t,.)|_{\partial\Omega}=0 \\
&(u(0),\partial_tu(0))=(u_0,v_0)\in V\times H
\end{aligned}
\right.
\end{equation}

\subsection{Some functional analysis preliminaries}

We work with a Hilbert space $\mathcal{H}:=V\times H$, equipped with the norm
$$ \|(f,g)^t\|_{\mathcal{H}}^2:=\|\nabla f\|_{L^2(\Omega)}^2+\|g\|_{L^2(\Omega)}^2.
$$
and denote $\Pi:L^2(\Omega)^N\rightarrow H$ be the orthogonal projector(Leray-projector) and $A=\Pi\Delta$ be the Stokes operator. We consider the operator:
\begin{equation}\label{operator}
\mathcal{A}= \left(
\begin{array}{cc}
0 & \mathrm{Id}  \\
A & -\Pi a
\end{array}
\right)
\end{equation}
with domain
$$ D(\mathcal{A})=(V\cap H^2(\Omega))\times V.
$$
In order to use semi-group theory, we first show that for some $\lambda>0,$ the operator
$(\mathcal{A}-\lambda)$ is invertible:
Take $(f,g)\in V\times H$, and consider the system
\begin{equation}\label{inverse}
\left\{
\begin{aligned}
&v-\lambda u=f\\
&Au-(\Pi a+\lambda)v=g \\
\end{aligned}
\right.
\end{equation}
We consider the bilinear form
$$ B(u_1,u_2)=\int_{\Omega}\nabla u_1\cdot\nabla u_2dx+\int_{\Omega}(\lambda^2+\lambda a(x))u_1\cdot u_2dx,
$$
defined on $V\times V$.
We then conclude from Lax-Milgram that for $\lambda>0$, there exists $u\in V$ such that for any
$w\in V$, we have
$$ B(u,w)=-\int_{\Omega}(g\cdot w+(a(x)+\lambda)f\cdot w)dx.
$$
Set $v=\lambda u+f$, we have solved the system \eqref{inverse} in weak sense. Standard regularity argument gives that
for $\lambda>0$.
$$ (\mathcal{A}-\lambda)^{-1}:\mathcal{H}\rightarrow D(\mathcal{A}),
$$
is a bounded, and $(\mathcal{A}-\lambda)^{-1}:\mathcal{H}\rightarrow\mathcal{H}$ is compact. Moreover, if $\lambda\in$ Spec$(\mathcal{A})$, we must have $\Re\lambda<0$. This will be clear in the proof of Proposition \ref{semi-obser1}.

However, since the operator $\mathcal{A}$ is not maximal dissapative, the Hille-Yoshida theorem is not applicable. A slightly general modification ensures the existence of semi-group $e^{t\mathcal{A}}$ which evolves the initial data in $D(\mathcal{A})$ and solves the equation \eqref{WStokes1} with more regular data.

For solution $u,\partial_t u$ to \eqref{WStokes1}, we consider the energy functional
$$E[u](t):=\frac{1}{2}\int_{\Omega}(|\partial_t u(t,x)|^2+|\nabla u(t,x)|^2)dx,
$$
and we calculate
\begin{equation*}
\begin{split}
\frac{d}{dt}E[u](t)&=\int_{\Omega}\partial_t u\cdot (\partial_t^2u-\Delta u)dx \\
&=-\int_{\Omega}\partial_t u\cdot \nabla p dx-\int_{\Omega}a(x)|\partial_t u|^2dx\\
&\leq 0,
\end{split}
\end{equation*}
thus
\begin{equation}\label{dissapation}
E[u](t)\leq E[u](s),\forall s\leq t.
\end{equation}
By density argument, we can solve \eqref{WStokes1} with initial data in $\mathcal{H}$ such that the energy dissipation
\eqref{dissapation} still holds.


\subsection{Observability and Stabilization}
In this section, we will prove the stabilization for damped system is equivalent to observability for undamped system. For this part, we follow closely in the appendix of \cite{BG} in which the authors have sketched the standard argument for damped wave equation. 

We first introduce the quantity
$$ D[u](T)=\int_0^T\int_{\Omega}a(x)|\partial_t u(t,x)|^2dxdt,
$$
and it quantifies the dissipation of the energy:
$$ E[u](T)=E[u](0)-D[u](T).
$$
\begin{proposition}\label{equivalence}
	The following assertions are equivalent:
	\begin{enumerate}
		\item Stabilization: There exists $C_0,\alpha>0$, such that for every solution $u\in C(\mathbb{R};V\cap H^2(\Omega))\cap C^1(\mathbb{R};V)$ to the damped system \eqref{WStokes1}, we have
		$$  E[u](t)\leq C_0E[u](0)e^{-\alpha t},\forall t\geq 0.
		$$
		\item Observability: There exists $C>0$ and $T>0$, such that, for every solution $v\in C(\mathbb{R};V\cap H^2(\Omega))\cap C^1(\mathbb{R};V)$ to the undamped system \eqref{WStokes-NDamp}, the observability inequality holds:
		$$ E[v](0)\leq C D[v](T).
		$$
	\end{enumerate}
\end{proposition}
\begin{proof}
	We first claim that the stabilization of damped system is equivalent to the observability of damped system.
	
	It is clear that
	$$ E[u](0)=E[u](t)-D[u](t).
	$$
	
	Let us first assume the stabilization of damped system. Argue by contradiction, suppose the observability of damped system does not hold. We first choose $T_0>0$ large enough such that $C_0e^{-\alpha T_0}<\frac{1}{2}$. We can select a sequence of solutions $(u_k)$ and such that
	$$  E[u_k](0)=1,D[u_k](T_0)\rightarrow 0,\textrm{ as } k\rightarrow\infty.
	$$
	We thus have
	$$ \frac{1}{2}>C_0e^{-\alpha T_0}\geq E[u_k](T_0)=E[u_k](0)-D[u_k](T_0)=1+o(1), \textrm{ as }k\rightarrow \infty,
	$$
	which leads to a contradiction.

	Let us now assume the observability for damped system, i.e.
	$$  E[u](0)\leq CD[u](T),
	$$
	We may assume that $C>1$, from the energy dissapation and observability, we have
	$$  E[u](2T)=E[u](0)-D[u](2T)\leq \left(1-\frac{1}{C}\right)E[u](0).
	$$
	For any $t>0$, we write $m=\left\lfloor\frac{t}{2T}\right\rfloor$, therefore we have
	$$ E[u](t)\leq E[u](m)\leq \left(1-\frac{1}{C}\right)^m E[u](0),
	$$
	after choosing $C_0,\alpha$ appropriately, we have the stabilization of damped system.
	
	Our second step is to justify the equivalence between observability of damped system \eqref{WStokes1} and undamped system \eqref{WStokes-NDamp}. To do this, we denote $u$ and $v$ be solutions of the damped and of the undamped system, respectively, with the same initial data at $t=0$. Let $w=u-v$, and simple calculations yield
	$$ \partial_t^2 w-\Delta w=-a\partial_t u-\nabla q,
	$$
	$$ \partial_t^2 w-\Delta w+a\partial_t w=-a\partial_t v-\nabla q,
	$$
	with some pressure function $q$.
	
	We calculate
	$$ \frac{d}{dt}E[w](t)=-\int_{\Omega}a(x)|\partial_t u|^2dx+\int_{\Omega}a(x)\partial_tu\cdot\partial_t vdx-\int_{\Omega}\partial_t w\cdot\nabla q dx,
	$$
	and the last term of left-hand side vanishes, thanks to $\partial_t w\in C(\mathbb{R};V).$ Thus we can write
	\begin{equation*}
	\frac{d}{dt}E[w](t)=-\int_{\Omega}a(x)|\partial_t u|^2dx+\int_{\Omega}a(x)\partial_tu\cdot\partial_t vdx
	\end{equation*} 
	or equivalently
	$$\frac{d}{dt}E[w](t)=-\int_{\Omega} a(x)\partial_tu\cdot\partial_twdx.
	$$
	
	Integrating the two expressions above and using the inequality of the type
	$$ |ab| \leq \epsilon |a|^2+C(\epsilon)|b|^2,\forall \epsilon>0, 
	$$
	one easily get
	\begin{equation}\label{w}
	E[w](T)\leq  B \textrm{ min }(D[u](T),D[v](T)),\forall T>0,
	\end{equation}
	where $B$ is another absolute constant.

	Now suppose we have observability for the damped system \eqref{WStokes1}, if $D[u](T)\leq D[v](T)$, the observability of undamped system \eqref{WStokes-NDamp} is trivial. Now assume that $D[u](T)>D[v](T)$, we deduce from  \eqref{w} that
	\begin{equation}
	\begin{split}
	E[v](0)&=E[u](0)\leq CD[u](T)\leq CD[W](T)+CD[v](T)\\&\leq C(E[w](T)+D[v](T))\leq CD[v](T).
	\nonumber
	\end{split}
	\end{equation}
	The derivation of observability from undamped system to the damped follows in the same way, and we omit the details.
\end{proof}

\begin{remark}
	Since the domain $D(\mathcal{A})$ is dense in $\mathcal{H}$ and the observability and energy decay only involves the $L^2$ norm of $\nabla u$ and $\partial_t u$, thus the same result of proposition \ref{equivalence} holds if we replace
	$u\in C(\mathbb{R};V\cap H^2(\Omega))\cap C^1(\mathbb{R};V)$ to $u\in C(\mathbb{R};V)\cap C^1(\mathbb{R};H)$.
\end{remark}


\subsection{Resolvent estimates and stabilization}
Recall that from the previous sections, the study of damped system \eqref{WStokes1} is equivalent to the project system
\begin{equation}\label{matrixevolution}
\begin{split}
&\frac{d}{dt}\left(\begin{array}{c}
u\\
\partial_t u
\end{array}\right)= \left(
\begin{array}{cc}
0 & \mathring{Id}  \\
A & -\Pi a
\end{array}
\right)\left(\begin{array}{c}
u\\
\partial_t u
\end{array}\right),\\
&\left(\begin{array}{c}
u\\
\partial_t u
\end{array}\right)\in C(\mathbb{R};V\times H).
\end{split}
\end{equation}
We will use the notation $U=(u,\partial_t u)^t$ in the sequel.

In this part, we follow almost the same way as in the appendix of \cite{BG}, only to pay attention to the changing of working spaces (appearance of the pressure term and divergence free structure). Moreover, we add some technical details which may seems standard to experts in analysis but not disposable for many
applied people.

From last section, we know that the observability of undamped system \eqref{WStokes-NDamp} is equivalent to the stabilization of damped system \eqref{WStokes1}, therefore we will concentrate ourselves to the study of stabilization of \eqref{WStokes1}. The following result is standard in semigroup theory:
\begin{proposition}\label{decay}
	Consider a semi-group $e^{tL}$ on a Hilbert space $\mathcal{X}$, with infinitesimal generator $L$ defined on $D(L)$. Then if there exists $C>0,\delta>0$ such that the resolvent of $L$, $(L-\lambda)^{-1}$ exists for $\Re\lambda\geq -\delta $ and satisfies
	$$ \forall \lambda\in \mathbb{C}^{\delta}:=\{z\in\mathbb{C}:\Re z>-\delta\},\|(L-\lambda)^{-1}\|_{\mathcal{L}(\mathcal{X})}\leq C.
	$$
	Then there exists $M>0$ such that for any $t>0,$
	$$  \|e^{tL}\|_{\mathcal{L}(\mathcal{X})}\leq Me^{-\frac{\delta t}{2} }.
	$$
\end{proposition}
We need a lemma from complex analysis. We temporarily use the convention of Fourier transform
$$ \widehat{u}(\tau)=\frac{1}{\sqrt{2\pi}}\int_{-\infty}^{\infty}e^{-it\tau}u(t)dt.
$$
\begin{lemma}\label{complex}
	Let $u,v$ be two continuous functions with support in $\mathbb{R}_+=(0,\infty)$. Assume moreover that $u,v\in L^2(\mathbb{R}_+)$ and $v$ has compact support. From Winer-Paley theory, we know that the Fourier transform $\widehat{v}$ admits a holomorphic extension to $\mathbb{C}$ and of exponential type. Given $a_0>0$, suppose that the Fourier transform $\widehat{u}$ is also holomorphic in $S_{a_0}=\{z\in\mathbb{C}:\Im z<a_0\}$ and satisfies
	$$  |\widehat{u}(z)|\leq C |\widehat{v}(z)|,\forall z\in S_{a_0}.
	$$
	Then for any $a<a_0$, $e^{at}u(t)\in L^2(\mathbb{R}_+)$ and
	$$ \int_0^{\infty}e^{2at}|u(t)|^2dt=\int_{-\infty}^{\infty}|\widehat{u}(\tau+ia)|^2d\tau.
	$$
\end{lemma}
\begin{proof}
	We first claim that
	\begin{equation}\label{v}
	\int_0^{\infty}e^{2at}|v(t)|^2dt=\int_{-\infty}^{\infty}|\widehat{v}(\tau+ia)|^2d\tau,\forall a\in\mathbb{R}.
	\end{equation}
	Indeed, since $v$ is compactly supported,
	$$  \widehat{v}(\tau+ia)=\frac{1}{\sqrt{2\pi}}\int_0^{\infty}e^{at}e^{-it\tau}v(t)dt
	$$
	which is analytic in $a$ and rapidly decreasing in $\tau$ for each fixed $a\in\mathbb{R}$. Thus one easily deduce from the Plancherel (or calculate the integral directly) that \eqref{v} is true.

	As a consequence, $\widehat{u}(.+ia)\in L^2(\mathbb{R})$ for each $a<a_0$. Notice also that $u\in L^2(\mathbb{R}_+)$, thus for each $a$ with $\Re a<0$, the formula
	\begin{equation}\label{u}
	\int_0^{\infty}e^{2at}|u(t)|^2dt=\int_{-\infty}^{\infty}|\widehat{u}(\tau+ia)|^2d\tau
	\end{equation}
	holds true and analytic with respect to $a$.
	In particular, $|\widehat{u}(z)|\leq C|\widehat{v}(z)|$,$z\in S_{a_0}$ implies that $\widehat{u}(\tau+ia)$ is rapidly decreasing in $\tau$ for each fixed $a<a_0$. For $z=\tau+ia$ with $a<a_0$, consider the integral
	$$ F(a,t)=\frac{e^{at}}{\sqrt{2\pi}}\int_{-\infty}^{\infty}e^{itz}\widehat{u}(z)d\tau=
	\frac{1}{\sqrt{2\pi}}\int_{-\infty}^{\infty}e^{it\tau}\widehat{u}(z)d\tau\in L^2(\mathbb{R}).
	$$
	From Cauchy integral theorem, we have that
	$$ F(a,t)=\frac{e^{at}}{\sqrt{2\pi}}\int_{-\infty}^{\infty}\widehat{u}(\tau)e^{it\tau}d\tau=e^{at}u(t)\in L^2(\mathbb{R}).
	$$
	From this, we conclude that \eqref{u} follows for all $a<a_0$.
\end{proof}
\begin{remark}
	In the previous lemma, the same results hold true if we replace $u,v$ to be Hilbert-space valued functions.
\end{remark}
\begin{proof}[Proof of proposition \ref{decay}]
	The basic tool to prove this proposition is the Fourier-Laplace transform in time variable. From the property of strongly continuous semi-group, we know that there exists $\omega_0>0$ such that (see \cite{Taylor2}) $$\|e^{tL}\|_{\mathcal{L}(\mathcal{X})}\leq e^{\omega_0 t},\forall t\geq 1.
	$$
	Take $u_0\in D(L)$, and pick a nonnegative cut-off $\chi\in C^{\infty}(\mathbb{R})$ such that $\chi\equiv 0,\forall t\leq 1$ and $\chi\equiv 1, \forall t>2.$  We define $u(t):=\chi(t)e^{tL-\omega t}u_0$ for some $\omega>\omega_0$ and thus $u\in L^{\infty}(\mathbb{R};\mathcal{X})$. Moreover, we have the equation
	$$  (\partial_t+\omega-L)u=\chi'(t)e^{tL-\omega t}u_0=:v(t).
	$$
	By taking Fourier transform we get
	$$   (i\tau+\omega-L)\widehat{u}=\widehat{v}(\tau).
	$$
	Since $v$ is compactly supported in positive axis in time variable, the $\widehat{v}(\tau)$ has a holomorphic and bounded extension in any domain of the form
	$$   S_{\alpha}=\{\tau\in\mathbb{C}:\Im \tau<\alpha\},\alpha>0.
	$$
	From the assumption on the resolvent, we deduce that $(i\tau+\omega-L)$ is invertible if $\tau\in S_{\delta+\omega}$ and thus $\widehat{u}(\tau)$ admits a bounded holomorphic extension to $S_{\delta+\omega}$ which satisfies
	$$  \|\widehat{u}(\tau)\|_{\mathcal{X}}\leq C\|\widehat{v}(\tau)\|_{\mathcal{X}}.
	$$
	Apply Lemma \ref{complex}, we deduce that
	\begin{equation}
	\begin{split}
	\int_{-\infty}^{\infty}\|e^{(\omega_0+\delta)t}u\|_{\mathcal{X}}^2dt&=\int_{-\infty}^{\infty}\|\widehat{u}(\xi+i(\omega_0+\delta))\|_{\mathcal{X}}^2d\xi \\
	&\leq C\int_{-\infty}^{\infty}\|\widehat{v}(\xi+i(\omega_0+\delta))\|_{\mathcal{X}}^2d\xi \\
	&\leq C \int_{-\infty}^{\infty}\|e^{(\omega_0+\delta)t}v\|_{\mathcal{X}}^2dt\leq C\|u_0\|_{\mathcal{X}}^2.
	\nonumber
	\end{split}
	\end{equation}
	We remark that one need use various types of Winer-Paley theorems to justify the above calculations, thanks to the fact that $u(t),v(t)$ is supported on $[1,\infty)$ and furthermore $v(t)$ has compact support.
	Take $\omega<\omega_0+\frac{\delta}{2}$ in the definition of $u$, we have that
	$$ \|e^{\frac{\delta t}{2}}e^{tL}u_0\|_{L^2(\mathbb{R}_{+};\mathcal{X})}\leq C_1\|u_0\|_{\mathcal{X}}.
	$$
	Thanks to the semi-group structure and uniform bound principal, we have that there exists $M_0>0,$ such that for any interval $I\subset (0,+\infty)$ of length $1$,
	$$ \sup_{t\in I,s>0,t+s\in  I}\frac{|f(t+s)|}{|f(t)|}\leq M_0.
	$$
	with $f(t)=\|e^{tL}u_0\|_{\mathcal{X}}$. Therefore, for any $T>0$,
	\begin{equation}
	\int_{T}^{T+1}e^{\delta t}|f(t)|^2dt\geq e^{\delta T}\min_{t\in[T,T+1]}|f(t)|^2.
	\nonumber
	\end{equation}
	Therefore,
	$$ |f(T+1)|^2\leq M_0^2\min_{t\in[T,T+1]}|f(t)|^2\leq e^{-\delta T}\int_{T}^{T+1}e^{\delta u}|f(t)|^2dt.
	$$
	This implies the exponential decay
	$$\|e^{tL}u_0\|_{\mathcal{X}}\leq Me^{-\frac{\delta t}{2}}\|u_0\|_{\mathcal{X}}.$$
\end{proof}

Now we can introduce the semi-classical observability
\begin{proposition}\label{semi-obser1}
	Assume that $a\in L^{\infty}(\Omega)\cap C^0(\ov{\Omega})$ and $\int_{\Omega}adx>0$. Then the stabilization of system \eqref{WStokes1} is implied by the following statement:
	$$ \exists h_0>0,C>0 \textrm{ such that }\forall 0<h<h_0,\forall (u,q,f)\in H^2(\Omega)\cap V\times H^1(\Omega)\times H
	$$
	solves the equation
	$$-h^2\Delta u-u+h\nabla q =f,
	$$
	we have
	\begin{equation}
	\label{semi-ob1}
	\|u\|_{L^2(\Omega)}\leq C\left(\|a^{1/2}u\|_{L^2(\Omega)}+\frac{1}{h}\|f\|_{L^2(\Omega)}\right).
	\end{equation}
\end{proposition}
For the proof, we need two lemmas.
\begin{lemma}
	Let $L$ be a linear operator on Hilbert space $\mathcal{X}$ with a compact resolvent $(L-\mathrm{Id})^{-1}$. Suppose the spectrum $\textrm{ Spec}(L)\subset\{z:\Re z<0\}$ and satisfies that for any $\sigma\in\mathbb{R}$, $L-i\sigma$ is invertible and satisfies the uniform bound
	$$  \sup_{\sigma\in\mathbb{R}}\|(L-i\sigma)^{-1}\|<\infty.
	$$
	Then there exists $\delta>0$, such that
	$$  \sup_{\lambda\in\mathbb{C}^{\delta}}\|(L-\lambda)^{-1}\|<\infty,
	$$
	where $\mathbb{C}^{\sigma}:=\{z\in\mathbb{C}:\Re z>-\sigma\}$ for any $\sigma\in\mathbb{R}$.
\end{lemma}
\begin{proof}
	Write
	$$ \sup_{\sigma\in\mathbb{R}}\|(L-i\sigma)^{-1}\|=C
	$$
	We denote $R(z)=(L-z)^{-1}$ for $z\in\rho(L):=\{z:z\in \mathbb{C}\setminus \textrm{ Spec}(L) \}$. Take $z_0\in \rho(L)$, we write
	$$ L-z=(L-z_0)(\mathrm{Id}+(L-z_0)^{-1}(z_0-z)),
	$$
	and for $|z-z_0|<\frac{1}{\|(L-z_0)^{-1}\|}$, we have
	$$\|R(z)\|\leq \|R(z_0)\|\sum_{n=0}^{\infty}|z-z_0|^n\|(L-z_0)^{-1}\|^n\leq R(z_0).
	$$
	Therefore, for $\lambda$ with $|\Re\lambda|\leq\delta$, where $0<\delta<\frac{1}{2C}$, we have $\|R(\lambda)\|\leq C$. To conclude, we only need show that there exists $C_1>0$, such that
	$$  \sup_{\Re z>\delta}\|(L-z)^{-1}\|\leq C_1.
	$$
	Consider the holomorphic equivalence $\varphi:\mathbb{C}^0\rightarrow \mathbb{D}$,$\psi=\varphi^{-1}$.
	$$ \varphi(z)=\frac{z-1}{z+1},\psi(\zeta)=\frac{1+\zeta}{1-\zeta},
	$$
	where $\mathbb{D}:=\{\zeta:|\zeta|<1\}$ be the unit disk. One easily verifies that the operator-valued function
	$$  \Phi(\zeta)=R(\psi(\zeta)):\mathbb{D}\rightarrow \mathcal{L}(\mathcal{X})
	$$
	is analytic and satisfies the Cauchy integral formula
	$$  \Phi(\zeta_0)=\frac{1}{2\pi i}\oint_{|\zeta|=1}\frac{\Phi(\zeta)}{\zeta-\zeta_0}d\zeta,\forall \zeta_0\in\mathbb{D}.
	$$
	Since $\textrm{dist }(\partial\mathbb{D},\varphi(\mathbb{C}^{-\delta}))\geq \epsilon_0>0$ for some $\epsilon_0$ depends only on $\delta$, we deduce that for any $z\in \mathbb{C}^{-\delta}$,
	$$ \|R(z)\|\leq \left\|\frac{1}{2\pi i}\oint_{|\zeta|=1}\frac{\Phi(\zeta)}{\zeta-\varphi(z)}d\zeta  \right\|\leq \frac{C}{\epsilon_0}.
	$$
\end{proof}

\begin{lemma}\label{uniquecontinuation}[Unique Continuation of Stoke Operator]
	Let $\sigma>0$ and $u\in V$ satisfies that
	$$ Au=\sigma^2 u.
	$$
	Then if $u|_{\omega}\equiv 0$, we must have $u\equiv 0$.
\end{lemma}
\begin{proof}
	It is equivalent to write
	$$ -\Delta u+\nabla p=\sigma^2u,\textrm{ div }u=0,u\in V,\int_{\Omega}pdx=0.
	$$
	Take divergence of the equation, we have $\Delta p=0$. The vanishing of $u$ in $\omega$ implies that $p$ equals to a constant in a component of $\omega$. Now since $\Omega$ is connected, the maximum principal implies that $p\equiv 0$ in $\Omega$. Therefore we have reduced to unique continuation of eigenfunction of Laplace operator, and this implies that $u\equiv 0$ in $\Omega$. 
\end{proof}

\begin{proof}[Proof of proposition \ref{semi-obser}]
	We need show that the semi-classical observability implies the stabilization.
	
	Note that the operator $(\mathcal{A}-\lambda)$ is invertible for any $\lambda>0$. One write
	$$  \mathcal{A}-z=(\mathrm{Id}+(1-z)(\mathcal{A}-\mathrm{Id})^{-1})(\mathcal{A}-1),\forall z\in \mathbb{C}.
	$$
	Since $\mathrm{Id}+(1-z)(\mathcal{A}-\mathrm{Id})^{-1}$ is Fredholm with index 0, we infer that $\mathcal{A}-z$ is invertible iff it is injective. In light of the previous lemmas and the Proposition \ref{decay}, we only have to prove the fact that
	\begin{equation*}
	\begin{split}
	\exists C>0,\textrm{ such that }&\forall \sigma\in\mathbb{R}, U\in D(\mathcal{A}),F\in V\times H, (\mathcal{A}-i\sigma) U=F \\ &\textrm{ implies }\|U\|_{\mathcal{H}}\leq C\|F\|_{\mathcal{H}}.
	\end{split}
	\end{equation*}
	
	We argue by contradiction. If it is not true then we can find sequences $(\sigma_n),(U_n),$ and $(F_n)$ such that
	$$ (\mathcal{A}-i\sigma_n) U_n=F_n, \|U_n\|_{\mathcal{H}}=1,\|F_n\|_{\mathcal{H}}<\frac{1}{n}.
	$$
	After extracting subsequences we may assume that $\sigma_n\rightarrow \sigma$, and we write
	$$ U_n=(u_n,v_n)^t,F_n=(f_n,g_n)^t.
	$$
	We have several cases to analyse, according to the limit value $\sigma$.
	\begin{enumerate}
		\item $\sigma=0$: In this case, we have $\mathcal{A}U_n=o(1)_{\mathcal{H}}$, which is equivalent to
		$$ v_n=o(1)_{H_0^1},Au_n-\Pi a v_n=o(1)_{L^2},
		$$
		thus $Au_n=o(1)_{L^2}$. Taking inner product with $u_n$ and integrating by part we have
		$$ \int_{\Omega}|\nabla u_n|^2dx=o(1).
		$$
		This contradicts to $\|U_n\|_{\mathcal{H}}=1$.

		\item $0<|\sigma|<\infty$:  In this case we have $\mathcal{A}U_n-i\sigma U_n=o(1)_{\mathcal{H}}$, or equivalently,
		$$ v_n-i\sigma u_n=o(1)_{H_0^1}, Au_n-(i\sigma+\Pi a)v_n=o(1)_{L^2}.
		$$
		Thanks to Poincar\'e inequality, we deduce that
		$$  Au_n-i\sigma(i\sigma+\Pi a)u_n=o(1)_{L^2}.
		$$
		Applying Rellich compact embedding theorem followed by extracting to suitable sub-sequences, we may assume that
		$$ u_n\rightarrow u,\textrm{ in } L^2(\Omega), u_n \rightharpoonup  u, \textrm{ in } V.
		$$
		Taking inner product with $u_n$, we have
		$$ -\int_{\Omega}|\nabla u_n|^2dx=-\sigma \int_{\Omega}|u_n^2|dx +i\sigma\int_{\Omega}a(x)|u_n|^2dx+o(1),
		$$
		which implies that $au\equiv 0 $ in $\Omega$. Thus we can conclude that $u$ is an eigenfunction of Stokes operator $A$ and vanishes in a non trivial open subset of $\Omega$. The unique continuation property for the system
		$$  -\Delta u+\nabla p=\sigma^2 u,\textrm{ div }u=0
		$$
		implies that $u\equiv 0$. As a consequence, we have that $u_n=o(1)_{H_0^1},v_n=o(1)_{L^2}$. This contradicts to the original assumptions.

		\item $ |\sigma|=\infty$: We only study the case $\sigma_n\rightarrow +\infty$ (the other one is obtained by considering $\ov{U_n}$). 
		
		Let $h_n=\sigma_n^{-1}$, and we deduce from the system $\mathcal{A}U_n-i\sigma_n U_n=o(1)_{\mathcal{H}}$:
		$$ h_n^2 Au_n+u_n-ih_n\Pi au_n=h_n^2\Pi af_n+ih_nf_n+h_n^2 g_n=o_{L^2}(h_n) $$
		$$h_n v_n-iu_n=h_nf_n=o(h_n)_{H_0^1},
		$$
		
		$$ h_n^2Av_n+v_n-ih_n\Pi av_n=ih_ng_n-h_n^2Af_n=o_{L^2}(h_n)+o_{H^{-1}}(h_n).
		$$
		Define the operator $P_h=h^2A+\textrm{Id}-ih\Pi a$ on $H$ with domain $H^2(\Omega)\cap V$, we have (dropping the subindex $n$ for the moment)
		$$ (P_hu|u)_{L^2}=\|u\|_{L^2(\Omega)}^2-\|h\nabla u\|_{L^2(\Omega)}^2-ih\|a^{1/2}u\|_{L^2(\Omega)}^2.
		$$
		Taking imaginary part, we have
		$$  \|a^{1/2}u\|_{L^2(\Omega)}^2\leq C\frac{\|P_h u\|_{L^2(\Omega)}\|u\|_{L^2(\Omega)}}{h}.
		$$
		Applying the semi-classical observability to the equation $$h^2Au+u=ih\Pi a u+\tilde{f}$$ with $\tilde{f}=o_{L^2}(h)$, we have
		\begin{equation}
		\begin{split}
		\|u\|_{L^2(\Omega)}^2\leq& C\left(\|a^{1/2}u\|_{L^2(\Omega)}^2+\frac{1}{h^2}(\|f\|_{L^2(\Omega)}^2+h^2\|a^{1/2}u\|_{L^2(\Omega)}^2)\right) \\
		\leq& \frac{C}{h}\|f\|_{L^2(\Omega)}\|u\|_{L^2(\Omega)}+\frac{C}{h^2}\|f\|_{L^2(\Omega)}^2.
		\end{split}
		\end{equation}
		This implies that
		$$ \|u_n\|_{L^2(\Omega)}\leq \frac{C}{h_n}\|f_n\|_{L^2(\Omega)}=o(1).
		$$
		To conclude, observe that $v_n$ satisfies
		$$ h_n^2Av_n+v_n=o_{H}(h_n)+
		o_{H^{-1}(\Omega)}(h_n^2),
		$$
		and we claim that if 
		$(h^2A+1)v=f_1+f_2$, then
		\begin{equation}\label{obser}
		\begin{split}
	 &\|v\|_{L^2(\Omega)}+\|h\nabla v\|_{L^2(\Omega)}\\ \leq
		&C\left(\|a^{1/2}v\|_{L^2}+
		\frac{\|f_1\|_{L^2(\Omega)}}{h}+\frac{\|f_2\|_{H^{-1}(\Omega)}}{h^2}\right).
		\end{split}
		\end{equation}
		Assume the claim for the moment, we thus have $\|h_n\nabla v_n\|_{L^2(\Omega)}=o(1)$, and
		$\|\nabla u_n\|_{L^2(\Omega)}=o(1)$, thanks to $u_n+ih_nv_n=ih_nf_n$. This contradicts to the original assumption.
		
		Now we turn to the proof of the claim. By density, \eqref{semi-ob1} still valid when $v\in V$. 
		Taking inner product of $v$ with $P_h v$, we have
		$$(P_hv|v)_{L^2}=\|v\|_{L^2(\Omega)}^2-\|h\nabla v\|_{L^2(\Omega)}^2-ih\|a^{1/2}v\|_{L^2(\Omega)}^2.
		$$
		Therefore, by taking real part and injecting \eqref{semi-ob1}, we have
		\begin{equation}\label{semi-ob2}
		\|h\nabla v\|_{L^2(\Omega)}+
		\|v\|_{L^2(\Omega)}
		\leq C\left(\|a^{1/2}v\|_{L^2(\Omega)}+\frac{\|P_h v\|_{L^2(\Omega)}}{h}\right).
		\end{equation}
	By taking real and imaginary part of $(P_h v|v)_{L^2}$, we have
		\begin{equation}
		\label{semi-ob3}
		\|a^{1/2}v\|_{L^2(\Omega)}^2
		\leq \frac{\|P_h v\|_{L^2(\Omega)}\|v\|_{L^2(\Omega)}}{h},
		\end{equation}
		 Substituting \eqref{semi-ob3} into 
		\eqref{semi-ob2}, we obtain that
		$$ \|h\nabla v\|_{L^2(\Omega)}^2 +\|v\|_{L^2(\Omega)}^2
		\leq C\left(\frac{\|P_h v\|_{L^2(\Omega)}\|v\|_{L^2(\Omega)}}{h}+\frac{\|P_h v\|_{L^2(\Omega)}^2}{h^2}\right),
		$$
		and this implies that
		$$  \|h\nabla v\|_{L^2(\Omega)} +\|v\|_{L^2(\Omega)}
		\leq C\frac{\|P_h v\|_{L^2(\Omega)}}{h}.
		$$
		Thus $P_h$ is bijective from $H^2(\Omega)\cap V$ to $H$ and hence invertible. From the fact that
		$$ P_h=(1+(2-ih\Pi a)(h^2A-1)^{-1})(h^2A-1),
		$$
		$P_h$ can be written as composition of a positive operator and a Fredholh operator of index 0. From the estimate above, we conclude that
		$$ \|P_h^{-1}\|_{L^2\rightarrow L^2}\leq \frac{C}{h},
		\|P_h^{-1}\|_{L^2\rightarrow H^1}\leq \frac{C}{h^2}.
		$$
		Now come back to the equation $(h^2A+1)v=f_1+f_2$. Taking $g\in H$, and letting $w=P_h^{-1}g$, we have 
		\begin{equation}
		\begin{split}
		(v|g)_{L^2}=&((h^2A+1) v|w)_{L^2}+ih(v|\Pi a w)_{L^2}\\
		=&(f_1+f_2|w)_{L^2}+ih(av|w)_{L^2}\\
		\leq &\|f_1\|_{L^2(\Omega)}\|P_h^{-1}g\|_{L^2()\Omega)}\\+&\|f_2\|_{H^{-1}(\Omega)}\|P_h^{-1}g\|_{V}+h\|av\|_{L^2(\Omega)}\|w\|_{L^2(\Omega)}\\
		\leq &C\left(\|a^{1/2}v\|_{L^2(\Omega)}+
		\frac{\|f_1\|_{L^2(\Omega)}}{h}+\frac{\|f_2\|_{H^{-1}(\Omega)}}{h^2}\right)\|g\|_{L^2(\Omega)}.
		\end{split}
		\end{equation}
		This completes the proof.
	\end{enumerate}

\end{proof}

\section{A priori Estimates for the quasi-mode system}
Now we consider the quasi-modes of Stokes system
\begin{equation}\label{Stokes-semi}
\left\{
\begin{aligned}
&-h_k^2\Delta u_k-u_k+h_k\nabla q_k=f_k,(u_k,f_k)\in (H^2(\Omega)\cap V)\times H,\\
&h_k\mathrm{div } u_k=0,\textrm{in} \ \Omega \\
\end{aligned}
\right.
\end{equation}
To simplify the notation, we drop the sub index $k$ and just keep the semi-classical parameter $h$ everywhere. Note that the functions $u,v,$ etc. should be understood as $u(h),v(h)$, etc. We fix the geometric assumption on the domain $\Omega\subset \mathbb{R}^d$ is smooth and connected and $\partial\Omega=\cup_{j=1}^N \Gamma_j$ with $\Gamma_j\cap\Gamma_k=\varnothing,i\neq k$ and each $\Gamma_j$ is smooth and connected.  

Now assume that
$$\|u\|_{L^2(\Omega)}=O(1),\|f\|_{L^2(\Omega)}=o(h).
$$
Taking inner product with $u$ and integrate by part, we have
$$ \|h\nabla u\|_{L^2(\Omega)}=O(1).
$$
One can always assume that $\int_{\Omega}qdx=0$, since $q\in L^2(\Omega)/\mathbb{R}$. From the regularity theory of steady Stokes system, (see \cite{Teman}, page 33), and Poincar\'{e} inequality, we have
$$ \|h^2\nabla^2 u\|_{L^2(\Omega)}=O(1),\|q\|_{L^2(\Omega)}=O(h^{-1}), \|h\nabla q\|_{L^2(\Omega)}=O(1).
$$

We now give some estimates on the trace. Write $q_0=q|_{\partial\Omega}$, 
\begin{lemma}
	$\|q\|_{L^2(\Omega)}=O(h^{-1}),
	\|q_0\|_{H^{1/2}(\partial\Omega)}=O(h^{-1}), \|q_0\|_{L^2(\partial\Omega)}=O(h^{-1}).$
\end{lemma}
\begin{proof}
	Since $q$ is harmonic function, then one can apply trace theorem  $H^s(\Omega)\rightarrow H^{s-1/2}(\partial\Omega)$ for any $s\in\mathbb{R}$. Hence the conclusions follows from these and interpolations.
\end{proof}

\begin{lemma}
	$h\partial_{\mathbf{\nu}}u|_{\partial\Omega}=(h\partial_{\mathbf{\nu}}u_{\pa},0)$, and
	$\|h\partial_{\mathbf{n}}u|_{\partial\Omega}\|_{L^2(\partial\Omega)}=O(1)$.
\end{lemma}
\begin{proof}
The first assertion follows from $h\mathrm{div }u=0$ and Dirichlet boundary condition, while we apply a multiplier method to prove the second. From the geometric assumption on $\Omega$, we can find a vector field $L\in C^1(\ov{\Omega})$ such that $L|_{\partial\Omega}=\nu$(see \cite{bookcontrol}, page 36). In global coordinate system, we write $ \displaystyle{L=L_j(x)\partial_{x_j}}.$ By using the equation, we have
\begin{equation*}
\begin{split}
\int_{\Omega}Lu\cdot fdx=&\int_{\Omega} Lu\cdot (-h^2\Delta u-u+h\nabla q)dx\\
-\int_{\Omega}Lu\cdot udx
=&-\int_{\Omega}L_j(x)\partial_{x_j} u^iu^idx\\
=&-\int_{\Omega}
\partial_{x_j}\left(L_j(x)u^i\right)u^idx+\int_{\Omega}\textrm{div }L(x)|u|^2dx\\
=&\int_{\Omega}L_j(x)u^i(x)\partial_{x_j}u^idx+
\int_{\Omega}\textrm{div }L(x)|u|^2dx\\
=&\int_{\Omega}Lu\cdot udx+
\int_{\Omega}\textrm{div }L(x)|u|^2dx,
\end{split}
\end{equation*}
thus
\begin{equation*}
\begin{split}
h\int_{\Omega}Lu\cdot \nabla qdx
&=-h\int_{\Omega}u^i \partial_{x_j}\left(L_j\partial_{x_i} q\right)dx\\
&=-h\int_{\Omega}u\cdot L(\nabla q)dx-h\int_{\Omega}
(\textrm{div }L(x))u\cdot\nabla qdx\\
&=-h\int_{\Omega}u\cdot[L,\nabla]qdx-h\int_{\Omega}\textrm{div }L(x)u\cdot\nabla qdx\\
&=O(1),
\end{split}
\end{equation*}
and
$\displaystyle{\int_{\Omega} Lu\cdot udx=-\frac{1}{2}\int_{\Omega}
	\textrm{div }L(x)|u|^2dx=O(1),}$
\begin{equation*}
\begin{split}
-h^2\int_{\Omega}Lu^i\Delta u^idx=&-h^2\int_{\partial\Omega} \left|\partial_{\nu}u^i\right|^2d\sigma +h^2\int_{\Omega}\nabla L(\nabla u^i,\nabla u^i)dx \\
&+h^2\int_{\Omega}
L_j(x)\partial^2_{x_jx_k} u^i\partial_{x_k} u^i\\
=&-h^2\int_{\partial\Omega} \left|\partial_{\nu}u^i\right|^2d\sigma +h^2\int_{\Omega}\nabla L(x)(\nabla u^i,\nabla u^i)dx
\\&+h^2\int_{\Omega}\partial_{x_j}\left(L_j\partial_{x_k}u^i\right)\partial_{x_k}u^idx
-h^2\int_{\Omega}\textrm{div }L(x)\nabla u^i\cdot\nabla u^i(x)dx,
\end{split}
\end{equation*}
\begin{equation*}
\begin{split}	h^2\int_{\Omega}\partial_{x_j}\left(L_j\partial_{x_k}u^i\right)\partial_{x_k}u^idx&=
h^2\int_{\partial\Omega}L\cdot\nu \left|\partial_{\nu} u^i\right|^2d\sigma -h^2\int_{\Omega}L_j(x)\partial_{x_k} u^i
\partial^2_{x_jx_k} u^idx,
\end{split}
\end{equation*}
$$-h^2\int_{\Omega}Lu^i\Delta u^idx
=-\frac{h^2}{2}\int_{\partial\Omega}\left|\partial_{\nu}u^i\right|^2d\sigma+
\int_{\Omega}\nabla L(x)(h\nabla u^i,h\nabla u^i)dx-
\frac{h^2}{2}\int_{\Omega}\textrm{div }L(x)|\nabla u^i|^2dx.
$$
Observing that
$ \int_{\Omega}Lu\cdot fdx=o(1),
$
we have
$$ \int_{\partial\Omega}\left|h\partial_{\nu}u\right|^2d\sigma=O(1).
$$
\end{proof}

\begin{lemma}\label{press.norm}
	Under additional assumption that
	$$\|a^{1/2} u_k\|_{L^2(\Omega)}=o(1),
	$$
	after extracting to subsequences, we have $h_k\nabla q_k\rightharpoonup 0$ $L^2(\Omega)$ and $u_k\rightharpoonup 0$ weakly in $L^2(\Omega)$. Therefore from Rellich theorem, we have $hq\rightarrow 0$, strongly in $L^2(\Omega)$.
\end{lemma}
\begin{proof}
	We may assume that $h\nabla q\rightharpoonup r$, weakly in $L^2(\Omega)$, and Rellich theorem implies that $hq\rightarrow P$, strongly in $L^2(\Omega)$, and thus $\nabla P=r$, with the property $\int_{\Omega} P=0$.  Now we claim that $\Delta P=0$ in $\Omega$.
	
	Indeed, take any $\varphi\in C_0^{\infty}(\Omega)$,
	$$\int_{\Omega}\nabla P\cdot\nabla\varphi =\lim_{h\rightarrow 0}\int_{\Omega}h\nabla q\cdot\nabla\varphi=0.
	$$
	Now suppose $u_k\rightarrow U$, weakly in $L^2(\Omega)$, $w_k=h_k^2u_k\rightarrow W$, weakly in $H^2(\Omega)$, by taking the weak limit in the equation, we have
	$$ -\Delta W-U+\nabla P=0,\textrm{ in }L^2(\Omega).
	$$
	Notice that $a^{1/2}u_k\rightarrow 0,a^{1/2}w_k\rightarrow 0$, strongly in $L^2(\Omega)$, and this implies that
	$U|_{\omega}=W|_{\omega}=0$. Therefore, in a connect component $\omega'$ of $\omega$, we have
	$\nabla P\equiv 0$. However, $P$ is a harmonic function, then $P\equiv $const., thanks to the fact that $\Omega$ is connected. Note that $\int_{\Omega}P=0$, hence $P\equiv 0$. Moreover, from Rellich theorem that $w_k\rightarrow W$ strongly in $L^2(\Omega)$, and on the other hand $\|h_k^2u_k\|_{L^2(\Omega)}=o(1)$ we must have $W=0$. Therefore $U=0$.
\end{proof}


\section{Proof of the Observability Estimates}

In this part, we will prove the Proposition \ref{semi-obser} under the assumption in Theorem \ref{main} on $\Omega$ and $\omega$.

We argue by contradiction, suppose \eqref{semi-ob} is not true, we can then choose a sequence $(u_n,h_n,q_n,f_n)\in H^2(\Omega)\cap V\times \mathbb{R}_+\times H^1(\Omega)\times H$ satisfies equation 
\begin{equation}\label{obserequation}
-h_n^2\Delta u_n-u_n+h_n\nabla q_n=f_n
\end{equation}
 with the following properties:
$$ \|u_n\|_{L^2(\Omega)}=1,\|f_n\|_{L^2(\Omega)}=o(h_n), \|a^{1/2}u_n\|_{L^2(\Omega)}=o(1),n\rightarrow\infty.
$$
Up to extracting to subsequence, we can associate $(u_n,h_n)$ with a semi-classical defect measure $\mu$. Therefore we have $\omega\cap \pi ($supp$(\mu))=\emptyset$ from Corollary \ref{vanishing}, where we denote $\pi: T^*\ov{\Omega}\rightarrow \ov{\Omega}$ be the canonical projection.

Denote $\phi(s,\rho)$ be the globally defined generalized bicharacteristic flow, thanks to the geometric assumption that $\Omega$ has no infinite contact. Pick any point $\rho_0$ with $\pi(\rho_0)\in\omega$.For any time segment $[0,s_0]$, there are several situations: 

Either $\phi([0,s_0],\rho_0)\subset \Omega$, or there exist $\pi(\phi([0,s_0],\rho_0))\cap \partial\Omega\neq \emptyset,$ then from the assumption on $\Omega$, all points $\phi(s,\rho_0)$ with $s\in[0,s_0]$ and $\pi(\phi(s,\rho_0))\in \partial\Omega$ must lie in $\mathcal{H}\cup \mathcal{G}^{2,+}\cup \mathcal{G}^{2,-}\cup \bigcup_{k\geq 3}\mathcal{G}^k$. Now Theorem \ref{St} implies that
$$ \textrm{ supp }(\phi(s,\cdot)_{*}\mu)\subset \textrm{ supp }(\mu).
$$
Therefore, we have
$$ \phi([0,s_0],\rho_0)\cap \textrm{ supp }(\mu)=\emptyset.
$$

We now invoke the geometric control condition to deduce that
$$ \ov{\Omega}\subset\pi\left(\bigcup_{\rho_0\in\omega} \phi([0,T_0],\rho_0)
\right)$$
for some $T_0>0$ and thus $\mu=0$. This contradicts to the assumption that
$$ \int_{\Omega}|u_n(x)|^2dx=1.
$$


\section{Appendix}
We will derive the hyperbolic Stokes system \eqref{undampedWS} from certain limit procedure of Lam\'e system from elastic theory:
\begin{equation}
\left\{
\begin{aligned}
&\partial_t^2w-\mu\Delta w-(\lambda+\mu)\nabla\textrm{div} w=0, (t,x)\in[0,T]\times \Omega\\
&w(t,.)|_{\partial\Omega}=0 \\
&(w(0),\partial_tw(0))=(w_0,z_0)\in (H_0^1(\Omega)\times L^2(\Omega))^d
\end{aligned}
\right.
\end{equation}
where the solution $w(t,x)$ is vector-valued. 

Define 
$ u(t,x):=w(t/\sqrt{\mu},x)$, then we find that
$$ \partial_t^2 u-\Delta u-\frac{\lambda+\mu}{\mu}\nabla\textrm{div} u=0.
$$
We let $\epsilon=\frac{\mu}{\mu+\lambda}\ll 1$, in the case that $\lambda\gg \mu>0.$ Thus we obtain a family of equations
\begin{equation}
\left\{
\begin{aligned}
&\partial_t^2u_{\epsilon}-\Delta u_{\epsilon}+\nabla p_{\epsilon}=0, (t,x)\in[0,T]\times \Omega\\
&u_{\epsilon}(t,.)|_{\partial\Omega}=0 \\
&(u_{\epsilon}(0),\partial_tu_{\epsilon}(0))=(u_{0,\epsilon},v_{0,\epsilon})\in (H_0^1(\Omega)\times L^2(\Omega))^d
\end{aligned}
\right.
\end{equation}
where $p_{\epsilon}=-\frac{1}{\epsilon}\textrm{div}u_{\epsilon}$ and satisfies $\int_{\Omega}p_{\epsilon}dx=0$.

We make further assumption on the family of initial data $(u_{0,\epsilon},v_{0,\epsilon})$ so that 
$$ \|(u_{0,\epsilon},v_{0,\epsilon})-(u_0,v_0)\|_{H^1\times L^2}\leq C\epsilon
$$
for some divergence free data $(u_0,v_0)\in V\times H$. In particular, we have
$$ \|\mathrm{div }u_{0,\epsilon}\|_{L^2(\Omega)}\leq C\epsilon.
$$

From the well-posedness of Lame system, we have that $u_{\epsilon}\in C([0,T];H_0^1(\Omega)),\partial_t u_{\epsilon}\in
C([0,T];L^2(\Omega))$, and  $p_{\epsilon}\in C([0,T];L^2(\Omega))$. Moreover, we have the conservation of energy
$$ E[u_{\epsilon}]=\frac{1}{2}\int_{\Omega}\left(|\partial_t u_{\epsilon}|^2+|\nabla u_{\epsilon}|^2+\epsilon |p_{\epsilon}|^2\right)dx
$$
and therefore
$$ E[u_{\epsilon}]=\frac{1}{2}\int_{\Omega}\left(|u_{0,\epsilon}|^2+|v_{0,\epsilon}|^2+\frac{1}{\epsilon}|\mathrm{div }u_{0,\epsilon}|^2\right)dx.
$$
From this, we have, up to some subsequence of $(u_{\epsilon},\partial_t u_{\epsilon})$
\begin{equation}
\begin{split}
& \textrm{div}u_{\epsilon}\rightarrow 0,\textrm{ in } L^{\infty}([0,T];L^2(\Omega)),\\
& u_{\epsilon}*\rightharpoonup u, *\textrm{weakly in }L^{\infty}([0,T];H_0^1(\Omega)),\\
& \partial_tu_{\epsilon}*\rightharpoonup \partial_t u, *\textrm{weakly in }L^{\infty}([0,T];L^2(\Omega)).
\nonumber
\end{split}
\end{equation}
From the uniform bound of $\|\partial_t u_{\epsilon}\|_{L^{\infty}([0,T];L^2(\Omega))}$, apply Ascoli theorem, we have that
(up to some subsequence)
$$ u_{\epsilon}\rightarrow u, \textrm{ in } C([0,T];L^2(\Omega)).
$$
Using the equation, we conclude that $\|\nabla p_{\epsilon}\|_{L^{\infty}([0,T];H^{-1}(\Omega))}$ is uniformly bounded. Combine with the fact $\int_{\Omega}p_{\epsilon}=0$, we have that $\| p_{\epsilon}\|_{L^{\infty}([0,T];L^{2}(\Omega))}$ is uniformly bounded, thus up to some subsequence, we may assume that
$$ p_{\epsilon}*\rightharpoonup  p, *\textrm{weakly in }L^{\infty}([0,T];L^2(\Omega)).
$$
Now it is not difficult to verify that $(u,p)$ is a weak solution to \eqref{WStokesmain}. Moreover, $p$ satisfies the zero mean condition
$$ \int_{\Omega}pdx=0.
$$


\subsection*{Acknowledgement}
This article is a part of PhD thesis of the second author. The authors would like to thank professor Gilles Lebeau, the advisor of the second author, for his encouragement and fruitful suggestions.

\begin{center}

\end{center}
\end{document}